\documentclass[11pt]{article}
\usepackage{graphicx}
\usepackage{amssymb}

\newcommand{\D}{\displaystyle}

\newcommand{\al}{\alpha}

\newcommand{\myref}[1]{(\ref{#1})}
\newcommand{\mathd}{\mathrm{d}}

\newenvironment{proof}{\noindent\textbf{Proof\ }}{\hspace*{\fill}$\Box$\medskip}

\newtheorem{theorem}{\textbf{Theorem}}[section]

\newtheorem{remark}{\textbf{Remark}}[section]
\newtheorem{proposition}{\textbf{Proposition}}[section]
\begin{document}
\title{On Singularity Formation of a Nonlinear Nonlocal System}
\author{Thomas Y. Hou\thanks{Applied and Comput. Math, Caltech,
Pasadena, CA 91125. {\it Email: hou@acm.caltech.edu.}} \and 
Congming Li\thanks{Department of Applied Mathematics, University of
Colorado, Boulder, CO. 80309. {\it Email: cli@colorado.edu}} \and
Zuoqiang Shi\thanks{Applied and Comput. Math, Caltech,
Pasadena, CA 91125. {\it Email: shi@acm.caltech.edu.}} \and
Shu Wang\thanks{College of Applied Sciences, Beijing University of
Technology, Beijing 100124, China. {\it Email: wangshu@bjut.edu.cn}}
\and Xinwei Yu\thanks{Department of Mathematics, University of Alberta,
Edmonton, AB T6G 2G1, Canada. {\it Email: xinweiyu@math.ualberta.ca}}
}
\date{\today}
\maketitle

\begin{abstract}

We investigate the singularity formation of a nonlinear
nonlocal system. This nonlocal system is a simplified
one-dimensional system of the 3D model that was recently proposed
by Hou and Lei in \cite{HouLei09} for axisymmetric 3D 
incompressible Navier-Stokes equations with swirl.
The main difference between the 3D model of Hou and Lei and the 
reformulated 3D
Navier-Stokes equations is that the convection term is neglected
in the 3D model. In the nonlocal system we consider in this paper, we
replace the Riesz operator in the 3D model by the Hilbert transform.
One of the main results of this paper is that we prove rigorously
the finite time singularity formation of the nonlocal system for a large
class of smooth initial data with finite energy. We also prove 
the global regularity for a class of smooth initial data. Numerical 
results will be presented to demonstrate the asymptotically 
self-similar blow-up of the solution. The blowup rate
of the self-similar singularity of the nonlocal system 
is similar to that of the 3D model.

\end{abstract}

\textbf{Key words}: Finite time singularities, 
nonlinear nonlocal system, stabilizing effect of convection.

\section{Introduction}

The question of whether a solution of the 3D incompressible
Navier-Stokes equations can develop a finite time singularity from
smooth initial data with finite energy is one of the most outstanding
mathematical open problems \cite{Fefferman,MB02,Temam01}. 
A main difficulty in 
obtaining the global regularity of the 3D Navier-Stokes equations 
is due to the presence of the vortex stretching term, which has a 
formal quadratic nonlinearity in vorticity. So far, most regularity
analysis for the 3D Navier-Stokes equations uses energy estimates.
Due to the incompressibility condition, the convection term does 
not contribute to the energy norm of the velocity field or
any $L^p$ ($1 < p \leq \infty$) norm of the vorticity field.
In a recent paper by Hou and Lei \cite{HouLei09}, the authors
investigated the stabilizing effect of convection by constructing
a new 3D model for axisymmetric 3D incompressible Navier-Stokes
equations with swirl. This model preserves
almost all the properties of the full 3D Navier-Stokes
equations except for the convection term which is neglected. If 
one adds the convection term back to the 3D model, one would recover 
the full Navier-Stokes equations. They also presented numerical evidence 
which supports that the 3D model may develop a potential finite time 
singularity. They further studied the mechanism that leads to 
these singular events in the 3D model and how the convection term in 
the full Navier-Stokes equations destroys such a mechanism. 

In this paper, we propose a simplified nonlocal system for the 3D 
model proposed by
Hou and Lei in \cite{HouLei09}. The nonlocal system is derived by first
reformulating the 3D model of Hou and Lei as the following two-by-two
nonlinear and nonlocal system of partial differential equations:
  \begin{equation}
  \label{Hou-Lei}
    u_t = 2 uv + \nu \Delta u, \hspace{2em} v_t = (-\Delta )^{- 1} \partial_{z z} u^2 + \nu \Delta v,
  \end{equation}
where $u= u^\theta/r$, $v= \psi^\theta_z/r$, and 
$\Delta = \partial_z^2 + \partial_r^2 + \frac{3}{r} \partial_r$,
and $u^\theta$ is the angular velocity component and 
$\psi^\theta$ is the angular stream function respectively,
$r = \sqrt{x^2+y^2}$. By the partial regularity result for
the 3D model \cite{HouLeiCMP09}, which is an analogue 
of the well-known Caffarelli-Kohn-Nirenberg partial
regularity theory for the 3D incompressible Navier-Stokes
equations \cite{CKN82}, we know that the singularity
can only occur along the symmetry axis, i.e. the $z$-axis.
In order to study the potential singularity formation of
the 3D model, it makes sense to construct a simplified 
one dimensional nonlocal system along the $z$-axis. One 
obvious choice is to replace the Riesz operator 
$(-\Delta )^{- 1} \partial_z^2$ by the Hilbert transform $H$
along the $z$ axis, and replace $\Delta u$ by $u_{zz}$, 
$\Delta v$ by $v_{zz}$. This gives rise to our simplified
nonlocal system:
\begin{equation}
\label{model-full}
  u_t = 2uv + \nu u_{zz}, \hspace{2em} v_t = H(u^2) + \nu v_{zz},
\end{equation}
where $H$ is the Hilbert transform, 
\begin{equation}
  \left( H f \right) \left( x \right) = \frac{1}{\pi}\mbox{P.V.}\int_{-
  \infty}^{\infty} \frac{f \left( y \right)}{x - y} \mathd y .
\end{equation}
In our analysis, we will focus on the inviscid version of the
nonlocal system and relabel the variable $z$ as $x$:
\begin{equation}
\label{model}
  u_t = 2uv , \hspace{2em} v_t = H(u^2) ,
\end{equation}
with the initial condition
\begin{eqnarray}
\label{ini data}
u(t=0)=u_0(x),\quad v(t=0)=v_0(x).
\end{eqnarray}

Note that the 1D model (\ref{model-full}) is designed to capture the 
dynamics of the 3D model (\ref{Hou-Lei}) along the $z$-axis only.
Thus, its inviscid model (\ref{model}) does not enjoy the energy 
conservation property of the original model in the three-dimensional 
space.

One of the main results of this paper is that we prove rigorously
the finite time singularity formation of the nonlocal system for 
a large class of smooth initial data with finite energy. 
As we will demonstrate in this paper, the blowup rate of
the self-similar singularity of the nonlocal system
\myref{model}-\myref{ini data} is qualitatively similar to 
that of the 3D model.
The main result of this paper is summarized in the following theorem.
\begin{theorem}
\label{blowup compact}
Assume that the support of $u_0$ is contained in $(a,b)$ and
that $u_0, \; v_0 \in H^1$. Let $\phi(x)=x-a$ and
\begin{eqnarray}
C=4 \int_a^b\phi(x)u_0^2\,v_0\,\mathd x,\quad I_\infty=\int_0^{+
\infty} \frac{\mathd y}{\sqrt{y^3+1}}.\nonumber
\end{eqnarray}
If $C>0$, then the solution of the nonlocal system
\myref{model}-\myref{ini data} 
must develop a finite time singularity in the $H^1$ norm
no later than $\D T^*=\left(\frac{4C}{3\pi(b-a)^2}\right)^{-1/3}I_\infty$.
\end{theorem}
A similar result has been obtained for periodic initial data.

The analysis of the finite time singularity for this nonlocal system is 
rather subtle. The main technical difficulty is that this is a
two-by-two nonlinear nonlocal system. The key issue is under
what condition the solution $u$ has a strong alignment 
with the solution $v$ dynamically. If $u$ and $v$ have a strong
alignment for long enough time, then the right hand side of
the $u$ equation would develop a quadratic nonlinearity dynamically,
which will lead to a finite time blowup. Note that $v$ is coupled
to $u$ in a nonlinear and nonlocal fashion through the Hilbert transform.
It is not clear whether $u$ and $v$ will develop such a nonlinear
alignment dynamically. To establish such a nonlinear alignment,
we need to use the following important property of the Hilbert
transform:
\begin{proposition} 
\label{prop hilbert}
Let $\phi$ be a globally Lipschitz continuous function on $R$. 
For any $f\in L^p(R^1) \cap L^1(R^1)$ and $\phi f\in L^q(R^1)$
  with $\frac 1p+\frac 1q=1$, $1<p, \; q<+\infty$,
we have
\begin{equation}
\label{prop hil}
    \int_{-\infty}^{+\infty}\phi(x)f(x)Hf(x)\mathd x
    =\frac{1}{2\pi}\int_{-\infty}^{+\infty}\int_{-\infty}^{+\infty}
    \frac{\phi(x)-\phi(y)}{x-y}f(x)f(y)\mathd x\mathd y .
  \end{equation}
  \end{proposition}
Using this property,
we can identify an appropriate test function $\phi$ such that
the time derivative of $\int u^2 \phi dz $ satisfies a nonlinear 
inequality. This inequality implies a finite time blowup 
of the nonlocal system. 

Proposition \ref{prop hilbert} should be a well-known property in
the Harmonic Analysis literature. During the revision of our paper,
we found that an identity which can be used to derive the 
special case $\phi = x$ of Proposition \ref{prop hilbert} has
been used in \cite{Duoandikoetxea08}, see also a recent paper
\cite{LR09b} \footnote{We only learned about the work of \cite{LR09b}
after the presentation of our work at the PIMS workshop
on Hydrodynamics Regularity in August 2009.}. 
However, we have not been able to find a proof
for the general case stated in Proposition \ref{prop hilbert}
in the literature. For the sake of completeness, we provide
a proof of Proposition \ref{prop hilbert} in Section 2.
 
Another interesting result is that we 
prove the global regularity of our nonlocal system for a class of 
smooth initial data. Specifically, we prove the following theorem:
\begin{theorem}
Assume that $u_0, \; v_0 \in H^1$. Further we assume that
$u_0$ has compact support in an interval of size $\delta$
and $v_0$ satisfies the condition
$v_0 \le - 3$ on this interval. Then the $H^1$ norm of the
solution of the nonlocal system \myref{model}-\myref{ini data} 
remains bounded for all time as long as the following holds
  \begin{equation}
\label{cond-v0}
    \delta^{1 / 2}  \left( \|v_{0 x} \|_{L^2} + \frac{1}{3} \delta^{1 / 2} 
\|u_{0 x}
    \|_{L^2}^2 \right) < \frac{1}{4} .
  \end{equation}
Moreover, we have $\|u \|_{L^\infty} \le Ce^{-3 t}$,
$\| u \|_{H^1} \le C e^{-3 t}$, and $\| v \|_{H^1} \le C$
for some constant $C$ which depends on $u_0$, $v_0$, and $\delta$ only.
\end{theorem}

In order to study the nature of the singularities, we have 
performed extensive numerical experiments for the nonlocal system
with or without viscosity. Our numerical study shows that 
$\|u\|_{L^\infty} (t)$ and $\|v\|_{L^\infty}(t)$ develop a
finite time blowup with a blowup rate 
$O\left ( \frac{1}{T-t} \right )$, which is qualitatively
similar to that of the 3D model \cite{HouLei09}.
Our numerical results also indicate that the solution of the 
inviscid nonlocal system seems to develop a one-parameter 
family self-similar finite time singularity of the type:
\begin{eqnarray}
\label{u-ss-0}
u(x,t) &= &\frac{1}{T-t}\,U\left(\xi,t\right),\\
\label{v-ss-0}
v(x,t) & = &\frac{1}{T-t}\,V\left(\xi,t\right),\\
\xi &= &\frac{x-x_0(t)}{(T-t)^{1/2}\log(1/(T-t))^{1/2}},
\label{Scaling}
\end{eqnarray}
where $x_0(t)$ is the position at which $u(x,t)$ achieves its
maximum. 
The parameter that characterizes this self-similar blowup
is the rescaled speed of propagation of the traveling wave
defined as follows: 
\[
\lambda=\lim_{t\rightarrow T}\left((T-t)^{1/2}\frac{d}{dt}x_0(t)\right).
\]
Different initial data give different speeds of
propagation of the singularity. One of the interesting
findings of our numerical study is that by rescaling the 
self-similar variable $\xi$ by $\lambda^{-1}$, the different 
rescaled profiles corresponding to different initial
conditions all collapse to the same universal profile. 
We offer some preliminary analysis to explain this 
phenomenon. 

Our numerical results also show that there is a significant overlap 
between the inner region of $U$ and the inner region of $V$ 
where $V$ is positive. Such overlap persists dynamically and is
responsible for producing a quadratic nonlinearity in
the right hand side of the $u$-equation. The nonlinear
interaction between $u$ and $v$ produces a traveling wave
that moves to the right\footnote{If we change the plus sign in 
front of the Hilbert transform in the nonlocal system \myref{model-full} 
to a minus sign, 
the nonlocal system would produce a traveling wave that moves 
to the left.}. Such phenomenon seems quite
generic, and is qualitatively similar to that of the 3D 
model \cite{HouLei09}. The only difference is that the
3D model produces traveling waves that move along the symmetry
axis in both directions. It is still a mystery why the 
inviscid nonlocal system selects the scaling (\ref{Scaling}) 
with the 1/2 exponent and a logarithmic correction.
With the logarithmic correction, the viscous term can not 
dominate the nonlinear term $2uv$ in the equation. Indeed, 
when we add viscosity to the nonlocal system, we find that the viscous
solution still develops the same type self-similar finite time
blowup as that of the inviscid nonlocal system.  


We remark that Hou, Shi and Wang \cite{HSW09} have recently
made some important progress in proving the formation of
finite time singularities of the original 3D model of Hou 
and Lei \cite{HouLei09} for a class of smooth initial 
conditions with finite energy under some appropriate
boundary conditions. The stabilizing effect of convection has 
been studied by Hou and Li in a recent paper \cite{HouLi07} via
a new 1D model. Formation of singularities for
various model equations for the 3D Euler equations
or the surface quasi-geostrophic equation
has been investigated by Constantin-Lax-Majda \cite{CLM85},
Constantin \cite{Constantin86}, 
DeGregorio \cite{DeGregorio90,DeGregorio96}, 
Okamoto and Ohkitani \cite{Ohkitani05},
Cordoba-Cordoba-Fontelos
\cite{CCF05}, Chae-Cordoba-Cordoba-Fontelos \cite{CCCF05},
and Li-Rodrigo \cite{LR09}.

The rest of the paper is organized as follows. In Section 2, 
we study some properties of the nonlocal system. In Section 3, 
we establish the local well-posedness of the nonlocal system. Section 4
is devoted to proving the finite time singularity formation of 
the inviscid nonlocal system for a large class of smooth initial data 
with finite energy. We prove the global regularity of the 
nonlocal system for a class of initial data in Section 5. Finally, 
we present several numerical results in Section 6 to study the 
nature of the finite time singularities for both the inviscid 
and viscous nonlocal systems.

\section{Properties of the nonlocal system}

In this section, we study some properties of the nonlocal system.
First of all, we note that the nonlocal system has some interesting
scaling property. Specifically, for any constants $\alpha$ and $\beta$
satisfying $\alpha \beta > 0$, the nonlocal system  
\begin{equation}
    u_t = \alpha uv, \hspace{2em} v_t = \beta H u^2
\end{equation}
is equivalent to the system 
  \begin{equation}
    \tilde{u}_t = 2 \tilde{u}  \tilde{v} , \hspace{2em} \tilde{v}_t = H
    \tilde{u}^2
  \end{equation}
by introducing the following rescaling of the solution:
\begin{equation}
    u = \tilde{u} \left( x, \gamma t \right), \hspace{2em} v = \mu
    \tilde{v} \left( x, \gamma t \right) ,
\end{equation}
where $\gamma$ and $\mu$ are related to $\alpha$ and $\beta$ through the 
following relationship:
  \begin{equation}
    \gamma = \sqrt{\frac{\alpha \beta }{2}} ,\quad
\mu = \mbox{sgn} (\alpha ) \sqrt{\frac{2\beta }{\alpha} }.
  \end{equation}
Therefore, it is sufficient to consider the nonlocal system in 
the following form:
  \begin{equation}
\label{uv-eqn}
    u_t = 2uv, \hspace{2em} v_t = H u^2 .
  \end{equation}

Moreover, if we replace the second equation $v_t = H u^2 $
by $v= H u^2$ and define $w = u^2$, then our nonlocal system is 
reduced to the well-known Constantin-Lax-Majda model \cite{CLM85}:
\begin{equation}
w_t = 4 w H w  .
\end{equation}

Before we end this section, we present the proof of 
Proposition \ref{prop hilbert}.

\noindent
\begin{proof}{\bf of Proposition \ref{prop hilbert}}.
Denote $\D\widetilde{f_\epsilon}(x)=\frac{1}{\pi}\int_{|x-y|\ge \epsilon}\frac{
f(y)}{x-y}\mathd y$,
$\D F_\epsilon(x)=\phi(x)f(x)\widetilde{f_\epsilon}(x)$ and $\D\bar
f(x)=\sup_{\epsilon\ge 0}|\widetilde{f_\epsilon}(x)|$. It follows
from the singular integral theory of Calderon-Zygmund \cite{CZ56}
that 
$\D\widetilde{f_\epsilon}(x)\rightarrow Hf(x)$ a.e. $x\in R^1$ and
\[
\|\bar f\|_{L^p}\le C_p\|f\|_{L^p}.
\]
Therefore, we have
$F_\epsilon(x)\rightarrow \phi(x)f(x)Hf(x)$ a.e. $x\in R^1$ and
$|F_\epsilon(x)|\le G(x),$ where $G(x)=|\phi(x)f(x)|\bar f(x)$
satisfies
\begin{eqnarray}\|G(x)\|_{L^1}&\le&\|\bar f(x)\|_{L^p}\|\phi(x)f(x)\|_{L^q}\nonumber\\
&\le&C_p\|f(x)\|_{L^p}\|\phi(x)f(x)\|_{L^q}<+\infty.
\end{eqnarray} 
Using the Lebesgue Dominated Convergence Theorem, we have
\begin{eqnarray}
\int \phi(x)f(x)H(f)\mathd x 
& = & \lim_{\epsilon\rightarrow 0} 
\int f(x)\phi(x) \D\widetilde{f_\epsilon}(x) dx \nonumber \\
&=& \frac{1}{\pi}\lim_{\epsilon\rightarrow 0} \int f(x)\phi(x)\int_{|x-y|\ge \epsilon} \frac{f(y)}{x-y}\mathd y\mathd x .
\label{eqn-P1}
\end{eqnarray}
Note that 
\begin{eqnarray*}
\int |f(y)| \left (\int_{|x-y|\ge \epsilon}
\frac{|f(x)\phi(x)|}{|x-y|}\mathd x \right)\mathd y  
&\leq & \int |f(y)| \left (\int
\frac{2|f(x)\phi(x)|}{\epsilon+|x-y|}\mathd x \right )\mathd y \\
&\leq & 2
\|\phi(x)f(x)\|_{L^q} \|(\epsilon+|x|)^{-1}\|_{L^p}
\int |f(y)|\mathd y \\
&= & C\|f(y) \|_{L^1} \|\phi(x)f(x)\|_{L^q} < \infty ,
\end{eqnarray*}
for each fixed $\epsilon >0$ since $f \in L^1$, $\phi f \in L^q$ by 
our assumption, and $C \equiv \|(\epsilon+|x|)^{-1}\|_{L^p} < \infty$ for $p>1$.
Thus Fubini's Theorem implies that
\begin{equation}
\frac{1}{\pi} \int f(x) \phi(x) \int_{|x-y|\ge \epsilon}
\frac{f(y)}{x-y}\mathd y \mathd x =   
\frac{1}{\pi} \int \int_{|x-y|\ge \epsilon} f(x) \phi(x) 
\frac{f(y)}{x-y}\mathd y \mathd x,
\label{eqn-P2}
\end{equation}
for each fixed $\epsilon > 0$.
Furthermore, by renaming the variables in the integration, we can
rewrite $1/2$ of the integral on the right hand side of \myref{eqn-P2}
as follows:
\[
\frac{1}{2\pi} \int\int_{|x-y|\ge \epsilon} f(x)f(y) \frac{\phi(x)}{x-y}
\mathd y\mathd x = - 
\frac{1}{2\pi} \int\int_{|x-y|\ge \epsilon} f(x)f(y) \frac{\phi(y)}{x-y}
\mathd x\mathd y ,
\]
which implies that
\begin{equation}
\frac{1}{\pi} \int\int_{|x-y|\ge \epsilon} f(x)f(y) \frac{\phi(x)}{x-y}
\mathd y\mathd x =
\frac{1}{2\pi} \int\int_{|x-y|\ge \epsilon} f(x)f(y) \frac{\phi(x)-\phi(y)}{x-y}
\mathd x\mathd y .
\label{eqn-P3}
\end{equation}
Since $f \in L^1(R)$ and $\phi(x)$ is globally
Lipschitz continuous on $R$, it is easy to show that 
\[
f(x) f(y) \frac{\phi(x)-\phi(y)}{x-y} \in L^1(R^2).
\]
Using the Lebesgue Dominated Convergence Theorem, we have
\begin{eqnarray}
\frac{1}{2\pi}\lim_{\epsilon\rightarrow 0} \int\int_{|x-y|\ge \epsilon} f(x)f(y) \frac{\phi(x)-\phi(y)}{x-y}\mathd x\mathd y
=\frac{1}{2\pi}\int\int f(x)f(y) \frac{\phi(x)-\phi(y)}{x-y}\mathd x\mathd y .
\label{eqn-P4}
\end{eqnarray}
Proposition \ref{prop hilbert} now follows from \myref{eqn-P1}-\myref{eqn-P4}.
\end{proof}

We remark that Proposition \ref{prop hilbert} is also valid for
periodic functions. Recall that for periodic functions 
(with period $2 \pi$) the Hilbert transform takes the form:
\begin{equation}
  \left( H f \right) \left( x \right) = \frac{1}{2 \pi}\mbox{P.V.}\int_0^{2 \pi}
  f \left( y \right)\cot \left( \frac{x - y}{2} \right) \mathd y.
\end{equation}
For the sake of completeness, we state the corresponding
result for periodic functions below:

\begin{proposition}
\label{prop hilbert-2}
Let $\phi$ be a periodic Lipschitz continuous function with period
$2 \pi$. 
For any periodic function $f$ with period $2 \pi$ satisfying
$f\in L^p([0, 2\pi])$ and $\phi f\in L^q([0,2 \pi])$
  with $\frac 1p+\frac 1q=1$, $1<p, \; q<+\infty$,
we have
\begin{equation}
\label{prop hil -2}
    \int_{0}^{2 \pi}\phi(x)f(x)Hf(x)\mathd x
    =\frac{1}{4\pi}\int_0^{2 \pi}\int_0^{2\pi}(\phi(x)-\phi(y)) \cot 
\left ( \frac{x-y}{2} \right )f(x)f(y)\mathd 
x\mathd y .
  \end{equation}
  \end{proposition}

The proof of Proposition \ref{prop hilbert-2} goes exactly the
same as the non-periodic case. We omit the proof Here.

\begin{remark}
As we see
in the proof of Proposition \ref{prop hilbert}, the key is
to use the oddness of the kernel in the Hilbert transform.
The same observation is still valid here:
\begin{eqnarray*}
&&\frac{1}{2 \pi}\int \int_{[0, 2 \pi]^2,\;|x-y| > \epsilon} 
f(x)f(y) \phi(x)
\cot \left ( \frac{x-y}{2} \right )
\mathd y\mathd x \\
&& = - 
\frac{1}{2 \pi}\int \int_{[0, 2 \pi]^2,\;|x-y| > \epsilon} 
 f(x)f(y) \phi(y)\cot \left ( \frac{x-y}{2} \right ) \mathd x\mathd y ,
\end{eqnarray*}
by renaming the variables in the integration.
\end{remark}


\section{Local well-posedness in $H^1$}
In this section, we will establish the local well-posedness in Sobolev space $H^1$.
\begin{theorem} (Local well-posedness)
\label{wellposedness}
For any $u_0, v_0 \in H^1$, there exists a finite time 
$T=T\left(\|u_0\|_{H^1},\|v_0\|_{H^1}\right)>0$ 
such that the nonlocal system \myref{model}-\myref{ini data} has a
unique smooth solution, $u,v \in C^1\left([0,T); H^1\right)$
for $ 0 \leq t \leq T$. Moreover, if $T$ is the first time at which 
the solution of the nonlocal system ceases to be regular 
in $H^1$ and $T < \infty$, then the solution must satisfy the 
following condition:
\begin{eqnarray}
\label{cond blowup}
\int_0^T \left(\|u\|_{L^\infty}+\|v\|_{L^\infty}\right)\mathd t=+\infty.
\end{eqnarray}
\end{theorem}
\begin{remark}
We remark that the condition (\ref{cond blowup}) is an analogue 
of the well-known Beale-Kato-Majda blowup criteria for the
3D incompressible Euler equation \cite{BKM84}.
\end{remark}
\begin{proof}
To show local well-posedness, we write the system as an ODE in the Banach space $X:=H^1 \times H^1$: 
\begin{equation}
U_t = F(U),
\end{equation}
where $U = (u,v)$, $F(U) = (2 u v, H(u^2))$. As $H^1(\mathbb{R})$ is an algebra, $F$ maps any open set in $X$ into $X$, and furthermore $F$ is locally Lipschitz on $X$. Local well-posedness of \myref{model}-\myref{ini data} then follows from the standard abstract ODE theory such as Theorem 4.1 in \cite{MB02}. 

The blow-up criterion (\ref{cond blowup}) follows from the following a priori estimates.
Multiplying the $u$-equation by $u$ and the $v$-equation by $v$,
and integrating over $R$, we obtain
  \begin{equation}
    \frac{\mathd}{\mathd t} \int u^2 \mathd x = 4\int u^2 v \mathd x\le 4 \|v\|_{L^{\infty}}
    \int u^2\mathd x ,
  \end{equation}
and
  \begin{eqnarray}
    \frac{\mathd}{\mathd t} \int v^2\mathd x &=& 2 \int vH u^2\mathd x = - 2 \int \left( H v
    \right) u^2\mathd x \le 2\|u\|_{L^{\infty}}  \int \left| H v \right| u\mathd x\nonumber\\
    &\le& \|u\|_{L^{\infty}}  \left( \int u^2 \mathd x+ \int v^2 \mathd x\right) .
  \end{eqnarray}
  
Similarly, we can derive $L^2$ estimates for $u_x$ and $v_x$ as follows:
  \begin{eqnarray}
    \frac{\mathd}{\mathd t} \int u_x^2\mathd x & = & 4 \int (vu_x^2 + uv_x u_x)\mathd x
    \nonumber\\
    & \le & 4 \|v\|_{L^{\infty}}  \int u_x^2\mathd x + 2
    \|u\|_{L^{\infty}}  \int (u_x^2 + v_x^2) \mathd x , 
  \end{eqnarray}
and
  \begin{eqnarray}
    \frac{\mathd}{\mathd t} \int v_x^2 \mathd x& = & 4 \int v_x H \left( uu_x \right)\mathd x
    \nonumber\\
    & = & 4 \int \left( H v_x \right) uu_x \mathd x\nonumber\\
    & \le & 2\|u\|_{L^{\infty}}  \int (u_x^2 + v_x^2) \mathd x. 
  \end{eqnarray}
Summing up the above estimates gives
  \begin{equation}
\label{max-estimate}
    \frac{\mathd}{\mathd t} \left( \|u\|_{H^1}^2 +\|v\|_{H^1}^2 \right)
    \le C \left( \|u\|_{L^{\infty}} +\|v\|_{L^{\infty}} \right)  \left(
    \|u\|_{H^1}^2 +\|v\|_{H^1}^2 \right) .
  \end{equation}
We see that the regularity is controlled by the quantity
\begin{equation}
  \|u\|_{L^{\infty}} +\|v\|_{L^{\infty}} .
\end{equation}
If $\int_0^T (\|u\|_{L^{\infty}} +\|v\|_{L^{\infty}} ) dt < \infty$, 
then it follows from (\ref{max-estimate}) that $\|u\|_{H^1} +\|v\|_{H^1}$
must remain finite up to $T$. Therefore, if $T$ is the first time
at which the solution blows up in the $H^1$-norm, we must have
\begin{eqnarray}
\int_0^T \left(\|u\|_{L^\infty}+\|v\|_{L^\infty}\right)\mathd t=+\infty.
\end{eqnarray}
\end{proof}
\section{Blow up of the nonlocal system}

In this section, we will prove the main result of this paper,
that is the solution of the nonlocal system will develop a finite time
singularity for a class of smooth initial conditions with finite
energy. We will prove the finite time singularity of the nonlocal system
as an initial value problem in the whole space and 
in a periodic domain.

\subsection{Initial Data with Compact Support}

We first consider the initial value problem in the whole space
and prove the finite time blow up of the solution of
the nonlocal system \myref{model}-\myref{ini data} for a large
class of initial data $u_0$ that have compact support. 

For the sake of completeness, we will 
restate the main result below:
\begin{theorem}
\label{blowup compact}
Assume that the support of $u_0$ is contained in $(a,b)$ and
that $u_0, \; v_0 \in H^1$. Let $\phi(x)=x-a$ and
\begin{eqnarray}
C=4 \int_a^b\phi(x)u_0^2\,v_0\,\mathd x,\quad I_\infty=\int_0^{+\infty} \frac{\mathd y}{\sqrt{y^3+1}}.\nonumber
\end{eqnarray}
If $C>0$, then the solution of the nonlocal system \myref{model}-\myref{ini data} 
must develop a finite time singularity in the $H^1$ norm no later than
$\D T^*=\left(\frac{4C}{3\pi(b-a)^2}\right)^{-1/3}I_\infty$.
\end{theorem}
\begin{proof} 
By Theorem \ref{wellposedness}, we know that there exists a finite 
time $T=T\left(\|u_0\|_{H^1},\|v_0\|_{H^1}\right)>0$
such that the nonlocal system \myref{model}-\myref{ini data} has a
unique smooth solution, $u,v \in C^1\left([0,T); H^1\right)$
for $ 0 \leq t < T$. 
Let $T^*$ be the largest time such that the nonlocal system with initial 
condition $u_0$ and $v_0$ has a smooth solution in $H^1$. 
We claim that $T^* < \infty$. We prove this by contradiction. 

Suppose that $T^* = \infty$, i.e. that the nonlocal system has a globally
smooth solution in $H^1$ for the given initial condition $u_0$ and $v_0$. 
Using \myref{model}, we obtain
\begin{eqnarray}
(u^2)_{tt}=4 (u^2v)_t=8u_tuv+4 u^2v_t=4 (u_t)^2+ 4 u^2H(u^2) .
\end{eqnarray}
Multiplying $\phi(x)$ to both sides of the above equation and
integrating over $[a,b]$, 
we have the following estimate:
\begin{eqnarray}
\label{ori ineq}
\frac{d^2}{dt^2}\int_a^b\phi(x)u^2(x,t)\mathd x&=& 4
\int_a^b\phi(x)(u_t)^2\mathd x+ 4 \int_a^b\phi(x)u^2H(u^2)\mathd x\nonumber\\
&\ge&  4 \int_a^b\phi(x)u^2H(u^2)\mathd x \; .
\end{eqnarray}
Note that the support of $u(x,t)$ is the same as that of 
the initial value $u_0$. 
Proposition \ref{prop hilbert} implies that 
\begin{eqnarray}
\label{ori ineq-1}
\int_a^b\phi(x)u^2H(u^2)\mathd x &= & \int_{-\infty}^{\infty}
\phi(x)u^2H(u^2)\mathd x \nonumber\\
&=&\frac{1}{2\pi}\int_{-\infty}^{\infty}\int_{-\infty}^{\infty}
u^2(x,t)u^2(y,t)\frac{\phi(x)-\phi(y)}{x-y}\mathd x\mathd y\nonumber\\
&=&\frac{1}{2\pi}\left(\int_a^bu^2(x,t)\mathd x\right)^2 \; .
\end{eqnarray}
Combining \myref{ori ineq} with \myref{ori ineq-1}, we get
\begin{eqnarray}
\label{ori ineq-2}
\frac{d^2}{dt^2}\int_a^b\phi(x)u^2(x,t)\mathd x
\geq \frac{2}{\pi}\left(\int_a^bu^2(x,t)\mathd x\right)^2 .
\end{eqnarray}
As we can see, Proposition \ref{prop hilbert} plays an essential role
in obtaining the above inequality, which is the key estimate in
our analysis of the finite time singularity of the nonlocal system.

By the definition of $\phi$, we have the following inequality:
\begin{eqnarray}
\label{eqn-phi}
\int_a^bu^2(x,t)dx&\ge&\frac{1}{b-a}\int_a^b\phi(x)u^2(x,t)\mathd x .
\end{eqnarray}
Combining \myref{ori ineq-2} with \myref{eqn-phi},  
we obtain the following key estimate:
\begin{eqnarray}
\frac{d^2}{dt^2}\int_a^b\phi(x)u^2(x,t)\mathd x
\geq \frac{2}{\pi (b-a)^2} \left(\int_a^b\phi(x)u^2(x,t)\mathd x\right)^2
\end{eqnarray}
Denoting $F(t) = \int_a^b\phi(x)u^2(x,t)\mathd x$ we obtain the ODE inequality system
\begin{equation}
\label{eqn-Ftt}
F_{tt} \geq \frac{2}{\pi (b-a)^2} F^2,\quad F_t(0) = C>0,\quad F(0)=\int_a^b \phi u_0^2 > 0.
\end{equation}
Since $F_t(0) = C >0$, integrating (\ref{eqn-Ftt}) from 0 to $t$ gives
$F_t(t) > 0$ for all $t \geq 0$. Denote
$\widetilde{F}(t) \equiv F(t) - F(0)$. Then we have 
$\widetilde{F}(t) \geq 0 $ for $t\geq 0$, $\widetilde{F}_t > 0$ and 
$\widetilde{F} (0) =0$. Since $F(0) > 0$ and
$\widetilde{F}(t) \geq 0 $, it is easy to show 
that $\widetilde{F}$ satisfies the same differential inequality
(\ref{eqn-Ftt}) as $F$. Therefore we can set $F(0)=0$ in the 
following analysis without loss of generality. 

Multiplying 
$F_t$ to $F_{tt} \geq \frac{2}{\pi (b-a)^2} F^2$ and integrating in time, we obtain
\begin{eqnarray}
\label{dF}
\frac{dF}{dt}\geq \sqrt{\frac{4}{3\pi(b-a)^2}F^3+C^2}.
\end{eqnarray}
It is easy to see from the above inequality that $F$ must blow
up in a finite time. Define 
\begin{eqnarray*}
I(x)=\int_0^x \frac{\mathd y}{\sqrt{y^3+1}},\quad 
J=\left( \frac{3\pi(b-a)^2 C^2}{4} \right)^{1/3} .
\end{eqnarray*}
Integrating \myref{dF} in time gives
\begin{eqnarray}
\label{blowup-I}
I\left(\frac{F(t)}{J}\right)\geq \frac{Ct}{J}.
\end{eqnarray}
Observe that both $I$ and $F$ are strictly increasing functions,
and $I(x)$ is bounded for all $x > 0$ while the right hand
side of \myref{blowup-I} increases linearly in time.
It follows from \myref{blowup-I} that $F(t)$ must blow up no later than
\begin{eqnarray}
\label{eqn-T-comp}
T^*=\frac{J}{C}I_\infty=\left(\frac{4C}{3\pi(b-a)^2}\right)^{-1/3}I_\infty.
\end{eqnarray}
This contradicts with the assumption that the nonlocal system has a
globally smooth solution for the given initial condition $u_0$
and $v_0$. This contradiction implies that the solution of
the nonlocal system \myref{model}-\myref{ini data} must develop 
a finite time singularity in the $H^1$ norm no later than 
$T^*$ given by \myref{eqn-T-comp}.
This completes our proof of Theorem \ref{blowup compact}.
\end{proof}

\subsection{Periodic Initial Data}

In this subsection, we will extend the analysis of finite time
singularity formation of the nonlocal system to periodic initial data.
Below we state our main result:
\begin{theorem}
\label{blowup periodic}
We assume that the initial values $u_0, v_0$ are periodic functions 
with period $2\pi$ and  the support of $u_0$ is contained in 
$(a,b) \subset (0,2\pi)$ with $b-a < \pi$. Moreover, we assume that 
$u_0, \; v_0 \in H^1[0, 2\pi]$. Let $\phi (x)$ be a $2\pi$-periodic
Lipschitz continuous function with $\phi(x)=x-a$ on $[a,b]$,  and
\begin{eqnarray}
C=4 \int_a^b\phi(x)u_0^2\,v_0\,\mathd x,\quad I_\infty=\int_0^{+\infty} \frac{\mathd y}{\sqrt{y^3+1}}.\nonumber
\end{eqnarray}
If $C>0$, 
then the solution of the nonlocal system \myref{model}-\myref{ini data} 
must develop a finite time singularity in the $H^1$ norm no later
than $\D T^*=\left(\frac{4C\cos (\frac{b-a}{2} )}{3\pi(b-a)^2}\right)^{-1/3}I_\infty$.
\end{theorem}
\begin{proof} 
As in the proof of Theorem \ref{blowup compact}, we also prove this
theorem by contradiction. Assume that the nonlocal system with the
given initial condition $u_0$ and $v_0$ has a globally smooth 
solution in $H^1$.  
As before, by differentiating \myref{model} with respect to $t$, 
we obtain the following equation:
\begin{eqnarray}
(u^2)_{tt}=4(u_t)^2+4 u^2H(u^2) .
\end{eqnarray}
Multiplying $\phi(x)$ to both sides of the above equation,
integrating over $[0,2\pi]$ and using Proposition \ref{prop hilbert-2},
we obtain the following estimate:
\begin{eqnarray}
\frac{d^2}{dt^2}\int_a^{b}\phi(x)u^2(x,t)\mathd x&=&\frac{d^2}{dt^2}\int_0^{2\pi}\phi(x)u^2(x,t)\mathd x\nonumber\\&=&4\int_0^{2\pi}\phi(x)(u_t)^2\mathd x+4 \int_0^{2\pi}\phi(x)u^2H(u^2)\mathd x\nonumber\\
&\ge& 4\int_0^{2\pi}\phi(x)u^2H(u^2)\mathd x\nonumber\\
&=&\frac{1}{\pi}\int_0^{2\pi}\int_0^{2\pi}u^2(x,t)u^2(y,t)(\phi(x)-\phi(y))
\cot \left (\frac{x-y}{2} \right ) \mathd y\mathd x\nonumber\\
&=&\frac{1}{\pi}\int_a^{b}\int_a^{b}u^2(x,t)u^2(y,t)(x-y)\cot\left (
\frac{x-y}{2} \right ) \mathd y\mathd x\nonumber\\
&\geq&\frac{M}{\pi}\left(\int_a^{b}u^2(x,t)\mathd x\right)^2,
\label{key-inequ-period}
\end{eqnarray}
where $\D M=\min_{-(b-a)\le x\le b-a} x\cot (x/2)$. Since $b-a<\pi$, 
we have 
\begin{eqnarray}
M=\min_{-(b-a)\le x\le b-a} \frac{x\cos (x/2)}{\sin (x/2)}\ge \min_{-(b-a)\le x\le b-a} 2\cos (x/2)=2\cos \left ( \frac{b-a}{2} \right ) >0.
\end{eqnarray}
Now, following the same procedure as in the proof of Theorem 
\ref{blowup compact}, we conclude that the solution must blow up 
no later than
\begin{eqnarray}
\label{eqn-T-period}
T^*=\left(\frac{4MC}{6\pi(b-a)^2}\right)^{-1/3}I_\infty\le \left(\frac{4C\cos \frac{b-a}{2} }{3\pi(b-a)^2}\right)^{-1/3}I_\infty.
\end{eqnarray}
This contradicts with the assumption that the nonlocal system with the
given initial condition $u_0$ and $v_0$ has a globally smooth
solution. This contradiction implies that 
the solution of the nonlocal system \myref{model}-\myref{ini data}
must develop a finite time singularity in the $H^1$ norm
no later than $T^*$ given by \myref{eqn-T-period}.
This completes the proof of Theorem \ref{blowup periodic}.
\end{proof}

\begin{remark}
We can also prove the finite time blowup of a variant of
our nonlocal system
\begin{eqnarray}
\label{model-2}
u_t=2uv, \quad v_t=-H(u^2),
\end{eqnarray}
by choosing the test function $\phi(x)=b-x$.
It is interesting to note that while the solution of
\myref{model} produces traveling waves that propagate
to the right, the solution of \myref{model-2} 
produces traveling waves that propagate to the left.
\end{remark}

\begin{remark}
Our singularity analysis can be generalized to give 
another proof of finite time singularity formation of
the Constantin-Lax-Majda model without using the exact
integrability of the model. More precisely, we consider
the Constantin-Lax-Majda model:
\begin{eqnarray}
\label{CLM}
\left\{\begin{array}{ll}
u_t=uH(u),& \\
u(t=0)=u_0(x),& x\in \Omega.
\end{array}
\right.
\end{eqnarray}
By choosing $\phi(x)=x-a$ and following the same procedure
as in the proof of Theorem \ref{blowup compact}, we 
can show that if $u_0$ is smooth and
has compact support, $\mbox{supp}\;u_0=[a,b]$ 
and $u_0(x)>0$ on $(a,b)$, then the $L_1$ norm of the
solution of \myref{CLM} must blows up no later than 
\begin{eqnarray}
T^*=\frac{2\pi(b-a)^2}{\D \int_a^b\phi(x)u_0\,\mathd x} \; .
\end{eqnarray}
Below we will give a different and simpler proof of the finite 
time blowup for the Constantin-Lax-Majda model.

Multiplying $\phi(x)$ to both sides of equation 
\myref{CLM},
integrating over the support $(a,b)$, and using Proposition
\ref{prop hilbert}, we obtain
\begin{equation}
\frac{\mathd}{\mathd t}\int_a^b (x-a)u\mathd x=\int_a^b (x-a)uH(u)\mathd x=\frac{1}{2\pi}\left(\int_a^b u\mathd x\right)^2.
\end{equation}
As $ \int_a^b (x-a) u \mathd x \leq (b-a) \int_a^b u  \mathd x$ due to $ u \geq 0$ for $ x \in [a,b]$, setting $F(t) = \int_a^b (x-a) u \mathd x$ we have
\begin{equation}
F_t  \geq \frac{1}{2 \pi (b-a)^2} F^2,\quad F(0) = \int_a^b \phi(x) u_0 \mathd x > 0.
\end{equation}
This leads to 
\begin{equation}
F(t) \geq \frac{F(0)}{1 - tF(0)/2 \pi (b-a)^2 } ,
\end{equation}
which implies the finite-time blowup of $F$ 
no later than $T^*=\frac{2\pi(b-a)^2}{\int_a^b\phi(x)u_0\,\mathd x}$.

Similar result can be obtained for periodic initial data following the
same analysis of Theorem \ref{blowup periodic}.
\end{remark}

\section{Global regularity for a special class of initial data}

In this section, we will prove the global regularity of the solution
of our nonlocal system for a special class of initial data.
Below we state our main result in this section.
\begin{theorem}
\label{Global existence}
Assume that $u_0, \; v_0 \in H^1$. Further we assume that
$u_0$ has compact support in an interval of size $\delta$ 
and $v_0$ satisfies the condition
$v_0 \le - 3$ on this interval. Then the $H^1$ norm of the
solution of the nonlocal system \myref{model}-\myref{ini data}
remains bounded for all time as long as the following holds
  \begin{equation}
\label{cond-v0}
    \delta^{1 / 2}  \left( \|v_{0 x} \|_{L^2} + \frac{1}{3} \delta^{1 / 2} \|u_{0 x}
    \|_{L^2}^2 \right) < \frac{1}{4} .
  \end{equation}
  Moreover, we have $\|u \|_{L^\infty} \le Ce^{-3 t}$, 
$\| u \|_{H^1} \le C e^{-3 t}$, and $\| v \|_{H^1} \le C$
for some constant $C$ which depends on $u_0$, $v_0$, and $\delta$ only. 
\end{theorem}

\begin{proof}
Note that \myref{cond-v0} implies that 
$\delta^{1 / 2} \|v_{0x} \|_{L^2} < \frac{1}{4}$ which gives
$- 4 + 2 \delta^{1 / 2} \|v_{0x} \|_{L^2} < - 3.5$.
By using an argument similar to the local well-posedness analysis,
we can show that there exists $T_0 > 0$ such that 
$\|u\|_{H^1}$ and $\|v\|_{H^1}$ are bounded, 
$v < - 2$ on $\mbox{supp} (u)$, and
$ 2 \delta^{1 / 2} \|v_x \|_{L^2} < 1 $ for $ 0 \leq t < T_0$.

Let $\left[ 0, T \right)$ be the largest time interval on which 
$\|u\|_{H^1}$ and $\|v\|_{H^1}$ are bounded, and 
both of the following inequalities hold:
\begin{eqnarray}
\label{cond-v-1}
v < - 2 \quad \mbox{on} \quad \mbox{supp} (u) \quad \mbox{and} \quad
 2 \delta^{1 / 2} \|v_x \|_{L^2} <  1 .
\end{eqnarray}
We will show that $T = \infty$.
  
We have for $ 0 \leq t < T$ that
  \begin{equation}
    \frac{\mathd}{\mathd t} \int u_x^2 \mathd x= 4 \int (vu_x^2 + uv_x u_x) \mathd x \le - 8
    \int u_x^2 \mathd x +4 \|u\|_{L^{\infty}} \|v_x \|_{L^2} \|u_x \|_{L^2} .
  \end{equation}
Observe that $\mbox{supp}(u) = \mbox{supp} (u_0)$ for all times.
Let $\Omega = \mbox{supp} (u_0)$. 
 Since $\mbox{supp} (u)$ has length $\delta$, we can use the
 Poincar\'e inequality to get
  \begin{equation}
\label{u-max}
    \|u\|_{L^{\infty}} \le \delta^{1 / 2} \|u_x \|_{L^2(\Omega)}
     = \delta^{1 / 2} \|u_x \|_{L^2} .
  \end{equation}
  Therefore we obtain the following estimate:
  \begin{eqnarray}
    \frac{\mathd}{\mathd t} \|u_x \|_{L^2} &\le & -4 \|u_x \|_{L^2} +
    2\delta^{1 / 2} \|v_x \|_{L^2} \|u_x \|_{L^2} \nonumber\\
 &= & \left( - 4 +
    2\delta^{1 / 2} \|v_x \|_{L^2} \right) \|u_x \|_{L^2} < -
    3 \|u_x \|_{L^2} .
  \end{eqnarray}
  Thus we have for $0 \leq t < T$ that
  \begin{equation}
\label{estimate-ux}
    \|u_x \|_{L^2} \le \|u_{0 x} \|_{L^2} e^{-3 t} .
  \end{equation}
  On the other hand, we have that 
  \begin{eqnarray}
    \frac{\mathd}{\mathd t} \int v_x^2 \mathd x&=& 4 \int v_x H \left( uu_x \right)\mathd x \le 4 \|v_x \|_{L^2} \| u u_x\|_{L^2} \nonumber\\
&\le& 4\|u\|_{L^{\infty}} \|v_x
    \|_{L^2} \|u_x \|_{L^2} \le 4 \delta^{1 / 2} \|v_x \|_{L^2} \|u_x
    \|_{L^2}^2 ,
  \end{eqnarray}
where we have used the property that $ \| H(f)\|_{L^2} \leq \|f\|_{L^2}$
and the Poincare inequality \myref{u-max}.
  Now using \myref{estimate-ux}, we get
  \begin{equation}
    \frac{\mathd}{\mathd t} \|v_x \|_{L^2} \le 2 \delta^{1 / 2} \|u_x
    \|_{L^2}^2 \le 2 \delta^{1 / 2} \|u_{0 x} \|_{L^2}^2 e^{-6 t} .
  \end{equation}
  As a consequence, we obtain for $0 \leq t < T$ that
  \begin{equation}
\label{estimate-v0x}
    \|v_x \|_{L^2} \le \|v_{0 x} \|_{L^2} + \frac{1}{3} \delta^{1 / 2} \|u_{0 x}
    \|_{L^2}^2 .
  \end{equation}
  Now observe that at the left end of the support of $u$, $v_t = H u^2$ is
  always negative. Since $v_0 \leq - 3$ on the support of $u$,
we conclude that $v (x,t)\le - 3$ at the left end of the support of $u$
for all times.
Now, we apply the Poincar\'e inequality in the support of $u$ and
use \myref{estimate-v0x} to obtain
  \begin{equation}
v \le - 3 + \delta^{1 / 2} \|v_x \|_{L^2(\Omega )} \le - 3 + \delta^{1
    / 2}  \left( \|v_{0 x} \|_{L^2} + \frac{1}{3} \delta^{1 / 2} \|u_{0 x} \|_{L^2}^2
    \right) ,
  \end{equation}
  on $\mbox{supp} (u)$ for all $t \in \left[ 0, T \right)$.
  
Next, we perform $L^2$ estimates. We can easily show by 
using $v_t = H u^2$ that
\begin{eqnarray*}
\frac{1}{2} \frac{d}{dt} \int v^2 dx = \int v H u^2 dx  
  \leq  \| v\|_{L^2} \| u^2 \|_{L^2} 
 \leq  \| v\|_{L^2} \| u \|_{L^\infty}^2 \delta^{1/2},
\end{eqnarray*}
which gives
\[
\frac{d}{dt} \|v\|_{L^2} \leq \delta^{1/2} \| u \|_{L^\infty}^2.
\]
It follows from \myref{u-max} and \myref{estimate-ux} that
\begin{equation}
\label{u-max1}
\| u \|_{L^\infty} \leq \delta^{1/2} \|u_{0x}\|_{L^2} e^{-3t} .
\end{equation}
Therefore, we obtain 
\[
\frac{d}{dt} \|v\|_{L^2} \leq \delta \| u_{0x} \|_{L^2}^2 e^{-6t},
\]
which implies
\begin{equation}
\label{v-L2}
\|v\|_{L^2} \leq \|v_0\|_{L^2} + \frac{1}{6} \delta \| u_{0x} \|_{L^2}^2,
\end{equation}
for $ 0 \leq t < T$.

Similarly, using $v < -2 $ on the support of $u$, we can easily
show that
\begin{equation}
\label{u-L2}
\|u\|_{L^2} \leq \|u_0\|_{L^2} e^{-4t} ,
\end{equation}
for $ 0 \leq t < T$.

To summarize, we have shown that $\|u\|_{H^1} $ and
$\|v\|_{H^1}$ are uniformly bounded for $ 0 \leq t < T$, 
and
  \begin{equation}
    \|v_x \|_{L^2} \le \|v_{0 x} \|_{L^2} + \frac{1}{3} \delta^{1 / 2} \|u_{0 x}
    \|_{L^2}^2 ,
  \end{equation}
  and
  \begin{equation}
    v \le - 3 + \delta^{1 / 2} \|v_x \|_{L^2} \le - 3 + \delta^{1
    / 2}  \left( \|v_{0 x} \|_{L^2} + \frac{1}{3} \delta^{1 / 2} \|u_{0 x} \|_{L^2}^2
    \right) ,
  \end{equation}
on $\mbox{supp} (u)$ for $ 0 \leq t < T$.
  
By our assumption on the initial data, we have 
  \begin{equation}
    \delta^{1 / 2}  \left( \|v_{0 x} \|_{L^2} + \frac{1}{3} \delta^{1 / 2} \|u_{0 x}
    \|_{L^2}^2 \right) < \frac{1}{4} \; .
  \end{equation}
Therefore, we have proved that if
  \begin{equation}
\label{assumption-A}
    v < - 2 \mbox{ on } \mbox{supp} (u) \quad \mbox{ and } \quad  2 \delta^{1 /
    2} \|v_x \|_{L^2}  < 1 ,
  \end{equation}
  $ 0 \leq t < T $, then we actually have 
  \begin{equation}
    v \le - 2.75 \mbox{ on } \mbox{supp}(u) \quad \mbox{ and } \quad 
     2\delta^{1 / 2} \|v_x \|_{L^2} < 0.5 ,
  \end{equation}
$0 \leq t < T$. This implies that we can 
extend the time interval beyond $[0, T)$ so that 
\myref{assumption-A} is still valid. This contradicts 
the assumption that
$[0, T)$ is the largest time interval on which 
\myref{assumption-A} is valid. This
contradiction shows that $T$ can not be a finite number, i.e.
\myref{assumption-A} is true for all times. This in turn
implies that $\|u\|_{H^1}$ and $\|v\|_{H^1}$ are bounded for
all times. Moreover, we have shown that both $\| u\|_{L^\infty}$
and $\|u\|_{H^1}$ decay exponentially fast in time and 
$\|v\|_{H^1} $ is bounded uniformly for all times 
(see \myref{u-max1}, \myref{estimate-ux}, \myref{estimate-v0x},
\myref{v-L2} and \myref{u-L2}).
This proves Theorem \ref{Global existence}.
\end{proof}

\section{Numerical Results}

In this section, we perform extensive numerical experiments 
to study the nature of the singularities of the nonlocal system.
Our numerical results demonstrate convincingly that the
nonlocal system develops asymptotically self-similar singularities
in a finite time for both the inviscid and the viscous nonlocal
systems.

\subsection{Set-up of the Problem}

In our numerical study, we use the following nonlocal 
system without the factor of 2 in front of the nonlinear term $uv$
in the $u$-equation \footnote{As we have shown in Section 2, dropping 
this factor only changes the scaling of the solution.}:
\begin{eqnarray}
u_t&=&uv+\nu u_{xx},\\
v_t&=&H(u^2)+\nu v_{xx}.
\end{eqnarray}
We study the nonlocal system for two types of initial data.
The first type of initial data has compact support.
The second type of initial data is periodic. The nature of
the singularities for these two types of initial data is
very similar. In the case of periodic data, we can use FFT 
to compute the Hilbert transform. This enables us to perform
our computations with a very high space resolution.

Below we describe the initial data that we will use in our
numerical experiments. We choose three different initial conditions.
The first initial condition has compact support which lies
in $\Omega=[0.45,0.55]$ and $v_0 \equiv  0$. Within the
compact support $\Omega$, $u_0$ is given by
\begin{eqnarray}
\mbox{Initial Condition I:} \quad \quad
u_0= \exp\left (1-\left ( 1-\left(\frac{x-0.5}{0.05}\right)^2 \right )^{-1}\right ) \; \mbox{for}\; x\in\Omega,
\quad v_0=0.\nonumber
\end{eqnarray}
We call this Initial Condition I. The largest 
resolution we use for Initial Condition I is $N=16,384$.
The timestep is 
chosen to be $\Delta t=10^{-3}/\|u\|_{L^\infty}$ in order to
 resolve the maximum growth of $\|u\|_{L^\infty}$.

The last two initial conditions are periodic with period one.
They are given as follows:
\begin{eqnarray}
\begin{array}{lll}
\mbox{Initial Condition II:}&\quad u_0=2+\sin(2\pi x)+\cos(4\pi x),&\quad v_0=0,\\
\mbox{Initial Condition III:}&\quad \D u_0=\frac{1}{1.2+\cos(2\pi x)},&\quad v_0=0 .
\end{array}\nonumber
\end{eqnarray}
We call them Initial Condition II and Initial Condition III respectively.
The largest resolution that we use for these two periodic 
initial conditions is $N=262,144=2^{18}$, and the timestep is chosen
to be $\Delta t=10^{-3}/\|u\|_{L^\infty}$.

We use the fourth order classical Runge-Kutta method to 
discretize the inviscid nonlocal system in time. For the 
viscous nonlocal system, we consider only periodic initial 
data since the
solution will not have compact support anymore. In order
to remove the stiffness of the time discretization due
to the viscous term, we first apply Fourier transform to 
the nonlocal system to obtain
\begin{eqnarray}
\partial_t \widehat{u}(k,t) &=&\widehat{(uv)}(k,t) - \nu k^2 
\widehat{u}(k,t),\\
\partial_t \widehat{v}(k,t) &=&-i \mbox{sgn}(k) \widehat{(u^2)}(k,t) 
- \nu k^2 \widehat{v}(k,t),
\end{eqnarray}
where $\widehat{u}(k,t)$ is the Fourier transform of $u$ and 
$k$ is the wave number.
We then reformulate the viscous term as an integral factor
\begin{eqnarray}
\frac{\partial}{\partial t}(e^{\nu k^2 t} \widehat{u}(k,t)) 
&=& e^{\nu k^2 t} \widehat{(uv)}(k,t) ,\\
\frac{\partial}{\partial t}(e^{\nu k^2 t} \widehat{v}(k,t)) 
 &=&-i \mbox{sgn}(k) e^{\nu k^2 t} \widehat{(u^2)}(k,t) .
\end{eqnarray}
Now we apply the classical Runge-Kutta method to discretize
the above system in time. The resulting time discretization
method will be free of the stiffness induced by the viscous term.

For periodic initial data, we use the spectral method
to discretize the Hilbert transform by using the 
explicit formula $\widehat{H}(k) = -i \mbox{sgn}(k)$.
For initial data of compact support, we  
use the well-known alternating trapezoidal rule to discretize
the Hilbert transform which gives spectral accuracy. For the
sake of completeness, we describe the method below, see also 
\cite{Shelley92}. 
Let $x_j = j h$ the grid point and $h>0$ is the grid size. 
The alternating trapezoidal rule discretization of the 
Hilbert transform is given by the following quadrature:
\begin{equation}
H(f)(x_i) = \sum_{(j-i) odd} \frac{f_j}{x_i-x_j} 2h .
\end{equation}
Therefore, our numerical method has spectral accuracy 
in space and and fourth order accuracy in time. The
high order accuracy of the method plus high space 
resolution and adaptive time-stepping is essential 
for us to resolve the asymptotically self-similar
singular solution structure of the nonlocal system.

\subsection{Asymptotically self-similar blowup of the inviscid nonlocal system}

In the singularity analysis, we have proved that the nonlocal system must
develop a finite time singularity for a large class of initial
data. However, the singularity analysis does not tell us the
nature of the singularity. Understanding the nature of the 
singularity is the main focus of our numerical
study. Our numerical results show that for all three initial 
conditions we consider here, they develop asymptotically
self-similar singularities in a finite time. The numerical
evidence of self-similar singularities is quite convincing 
for all three initial data that we consider. As is the case
for the original 3D model, the mechanism of forming such 
self-similar blowup of the nonlocal system is due to the fact that
we neglect the convection term in our model. As it is 
demonstrated in \cite{HouLi07,HouLei09}, the convection term 
tends to destroy the mechanism for generating the finite time 
blowup in the 1D or 3D model. Indeed, a recent numerical study 
shows that the 3D incompressible Euler equation does not seem 
to grow faster than double exponential in time \cite{HL06,Hou09}. 

We use the following asymptotic singularity form fit to 
predict the singularity time and the blowup rate:
\begin{eqnarray}
\label{umax fit form}
\|u\|_{L^\infty}=\frac{C}{(T-t)^\al},
\end{eqnarray}
where $T$ is the blowup time.
We find that near the singularity time, the 
inverse of $\|u\|_{L^\infty}$ is almost a perfect linear 
function of time, see Figure \ref{umax compact}.
\begin{figure}
\begin{center}
\includegraphics[width=0.6\textwidth]{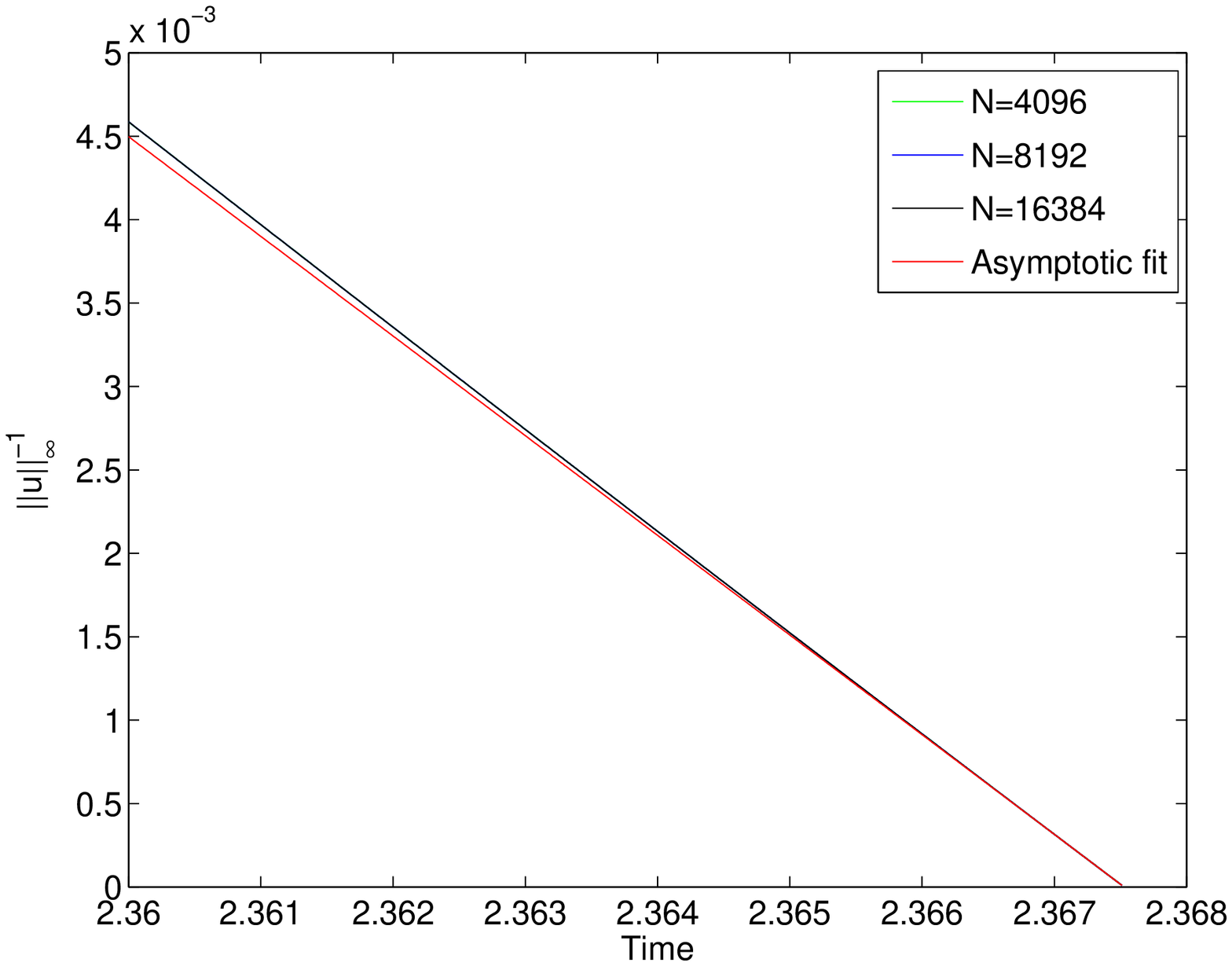}
\includegraphics[width=0.6\textwidth]{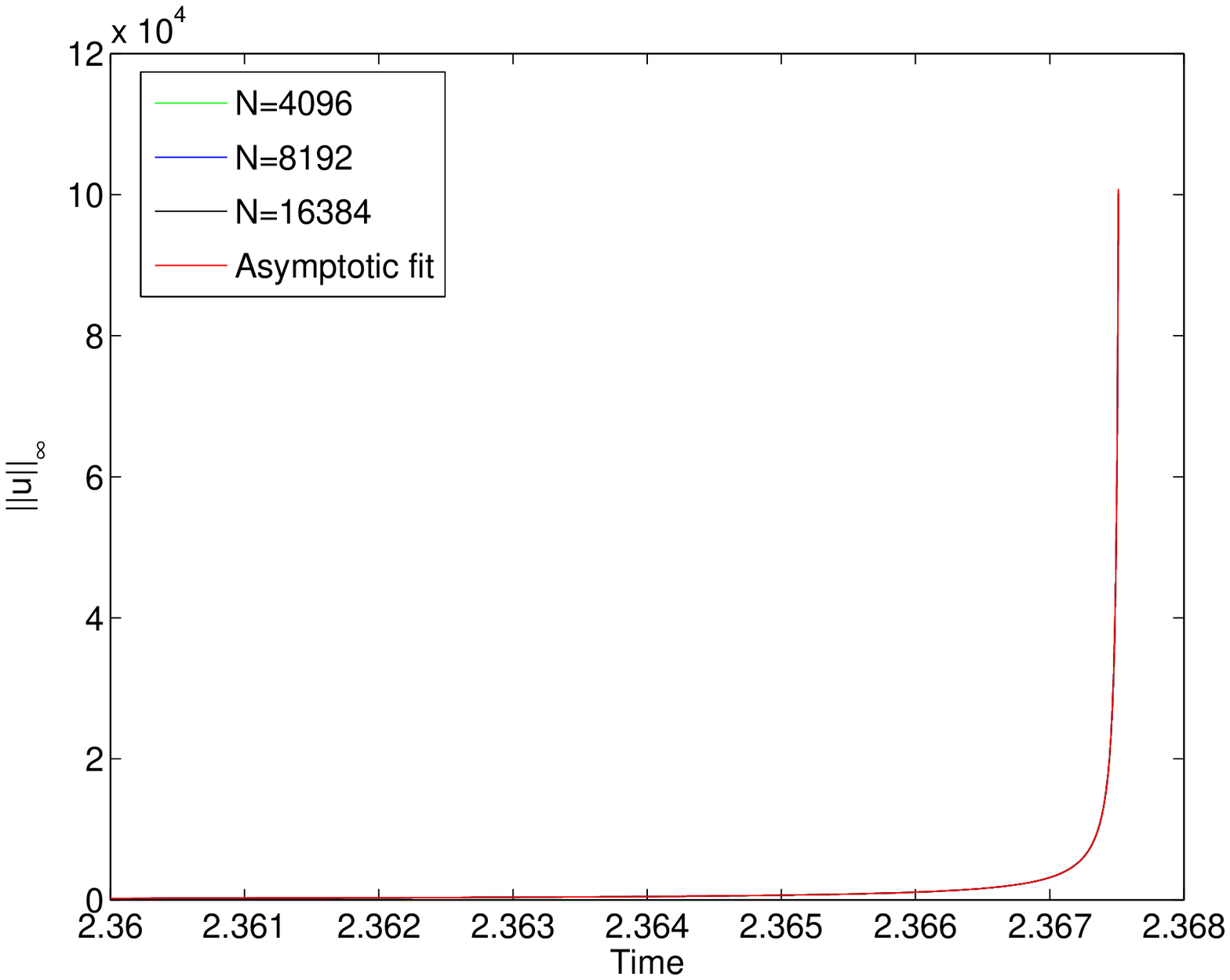}
\end{center}
\caption{Top: The inverse of $\|u\|_{L^{\infty}}$ (black) versus 
the asymptotic fit (red) for Initial Condition I with $\nu =0$. 
The fitted blowup 
time is $T=2.36752830915169$ and the scaling constant is
$C=1.67396437016231$. Bottom: 
$\|u\|_{L^{\infty}}$ (black) versus the asymptotic fit (red) 
for Initial Condition I.}
\label{umax compact}
\end{figure}
To obtain a good estimate for the singularity time, we perform
a least square fit for the inverse of $\| u\|_{L^\infty}$.  
We find that $\alpha =1$ gives the best fit.
The same least square fit also determines the potential
singularity time $T$ and the constant $C$.

To confirm that the above procedure indeed gives a good fit
for the potential singularity, we plot $\|u\|_{\infty}^{-1}$
as a function of time with a sequence of increasing
resolutions against the asymptotically form fit for the
three initial conditions we consider here.
In Figure \ref{umax compact}, we perform
such comparison for Initial Condition I
with a sequence of increasing resolutions from $N=4096$ to 
$N=16384$.
We can see that the agreement between the computed solutions 
and the asymptotically fitted solution is excellent as the
time approaches the potential singularity time. In the lower
box of Figure \ref{umax compact}, we plot
$\|u\|_{\infty} $ computed by our adaptive method
against the form fit $C/(T-t)$ with 
$T=2.36752830915169$ and $C=1.67396437016231$.
The computed solutions and the asymptotically fitted solution 
are almost indistinguishable. This asymptotic blowup rate
is qualitatively similar to that of the 3D model \cite{HouLei09}.

We have also performed a similar comparison between
the computed $\|u\|_{L^\infty}$ and the asymptotically
fitted solution for Initial Conditions II and III in
Figures \ref{umax ini1} and Figures \ref{umax ini2}
respectively. For these two periodic initial conditions,
we can afford even higher resolutions 
ranging from $N=2^{14}$ to $N=2^{18}$. Again, we observe
excellent agreement between the computed solutions
and the asymptotically fitted singular solution. 

\begin{figure}
\begin{center}
\includegraphics[width=0.6\textwidth]{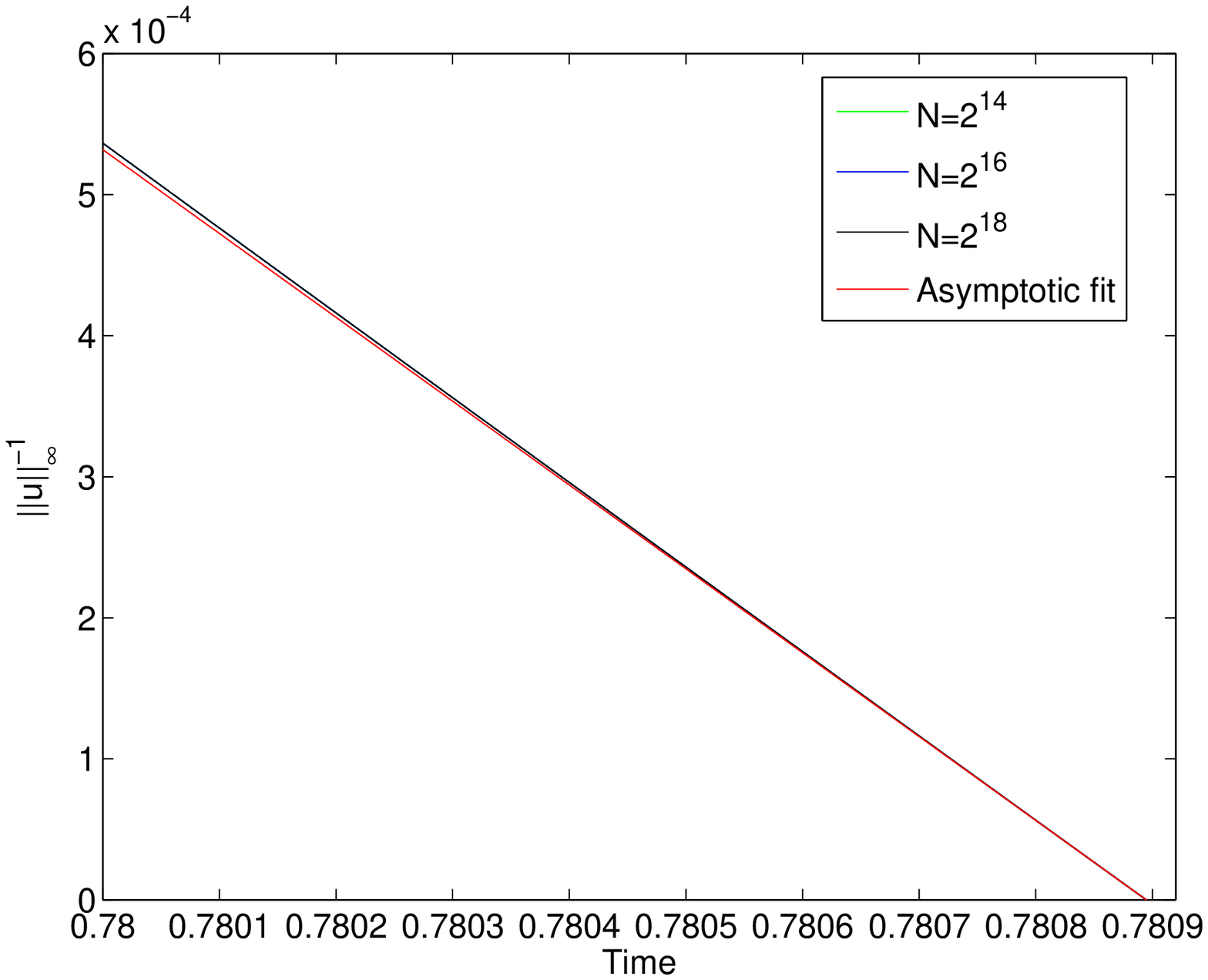}
\end{center}
\caption{The inverse of $\|u\|_{L^{\infty}}$ (black) versus the asymptotic fit (red) for Initial Condition II with $\nu = 0$. The fitted 
blowup time is $T=0.780894805082166$ and the scaling constant is
$C=1.68253514799506$. }
\label{umax ini1}
\end{figure}

\begin{figure}
\begin{center}
\includegraphics[width=0.6\textwidth]{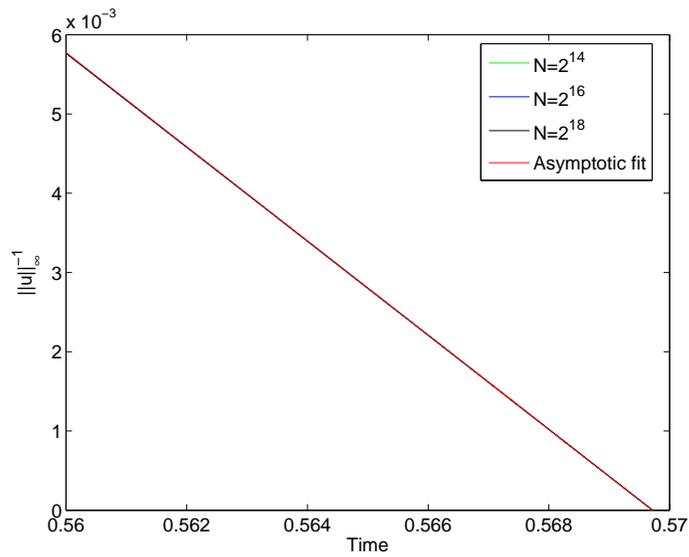}
\end{center}
\caption{The inverse of $\|u\|_{L^{\infty}}$ (black) 
versus the asymptotic fit (red) for Initial Condition III 
with $\nu = 0$. The fitted blowup time is 
$T=0.569719056780405$ and the scaling constant
is $C=1.68293676812485$.}
\label{umax ini2}
\end{figure}

After we obtain an estimate for the singularity time, we
can use it to look for a dynamically rescaled profile 
$U(\xi,t), V(\xi,t)$ near the singularity of the form 
\begin{eqnarray}
\label{u-self-similar}
u(x,t)=\frac{1}{T-t}\,U\left(\frac{x-x_0(t)}{(T-t)^\beta},t\right), 
\quad \mbox{as} \; t\rightarrow T,\\
\label{v-self-similar}
v(x,t)=\frac{1}{T-t}\,V\left(\frac{x-x_0(t)}{(T-t)^\beta},t\right),\quad 
\mbox{as} \;t\rightarrow T,
\end{eqnarray}
where $T$ is the predicted blowup time in the singularity form 
fit \myref{umax fit form}, $\beta$ is a parameter to be determined, 
and $x_0(t)$ is the location in which $|u|$ achieves its 
global maximum at $t$.

Again, we use a least square fit to determine $\beta$. 
Our numerical study indicates that $\beta=\frac{1}{2}$ with
a logarithmic correction. More precisely, we find that
the dynamically rescaled variable $\xi$ has the form:
\begin{eqnarray}
\xi=\frac{x-x_0(t)}{(T-t)^{1/2}\log(1/(T-t))^{1/2}}.
\label{xi-ss-form}
\end{eqnarray}
In terms of this rescaling variable $\xi$, we define
the dynamically rescaled profiles $U(\xi,t)$ and $V(\xi,t)$ 
through the following relationship:
\begin{eqnarray}
\label{u-self-similar-2}
u(x,t)&=&\frac{1}{T-t}\,U\left(\xi,t\right),\\
\label{v-self-similar-2}
v(x,t)&=&\frac{1}{T-t}\,V\left(\xi,t\right).
\end{eqnarray}
In Figure \ref{profile compact}, we plot the
self-similar profiles $U$ and $V$ at three 
different times for Initial Condition I. We can see
that the rescaled profiles for these three different times
agree with one another very well. From Figure \ref{profile compact},
we can see that there is a significant overlap between the inner
region of $U$ and the inner region of $V$ where $V$ is
positive. Such overlap persists dynamically and is 
responsible for producing a quadratic nonlinearity in
the right hand side of the $u$-equation, which has the
form $2 uv$. On the other hand, we observe that the
position at which $u$ achieves its global maximum is
not in phase with the position at which $v$ achieves its
global maximum. In fact, the positive part of $V$ always
moves ahead of $U$. This is a consequence of the property
of the Hilbert transform. As a result, the nonlinear
interaction between $u$ and $v$ produces a traveling wave
that moves to the right. Such phenomena seem quite
generic. We observe the same phenomena for all three
initial conditions for both the inviscid and the viscous
models. This phenomenon is also qualitatively similar to
that of the 3D model \cite{HouLei09}.

The strong alignment between the rescaled profile of $u$
and $v$ is the main mechanism for the solution of the 
nonlocal system to develop an asymptotically self-similar 
singularity in the form given by \myref{xi-ss-form} and 
\myref{u-self-similar-2}-\myref{v-self-similar-2}.
We observe essentially the same phenomena for Initial 
Conditions II and III, see Figure \ref{profile compare invis}.

\begin{figure}
\begin{center}
\includegraphics[width=0.6\textwidth]{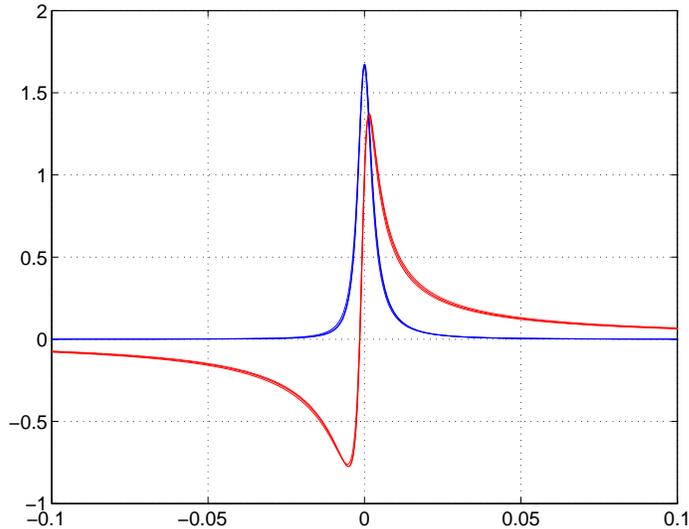}
\end{center}
\caption{Rescaled profiles $U$ and $V$ for Initial Condition I
with $\nu = 0$ at three different times:
$t= 2.36710445318745,\; 2.36743324526419$ and  
2.36750705502071, the corresponding maximum values of $u$ are 
3948, 17617 and 78422 respectively. 
Blue: profile of $u$; Red: profile of $v$.}
\label{profile compact}
\end{figure}

%

It is interesting to see how the different rescaled profiles
corresponding to different initial conditions are related to one 
another. 
In Figure \ref{profile compare invis} (top), we put three profiles 
from three different initial conditions together. The profile from 
Initial Condition III is the widest while the profile from 
Initial Condition II is narrower than that from Initial Condition III.
The profile from Initial Condition I is the narrowest of the three
initial conditions. But what is amazing is that they can 
match each other very well by rescaling the variable
variable $\xi$. To match the three rescaled profiles, we keep the 
profile from Initial Condition III unchanged. In order to match the 
profile from Initial Condition III, we change the profile from 
Initial Condition II by rescaling $\xi\rightarrow \xi/1.58$, and 
change the profile from Initial Condition I by rescaling 
$\xi\rightarrow \xi/19.5$. As we can see from Figure 
\ref{profile compare invis} (bottom), the three rescaled 
profiles is almost indistinguishable.

\begin{figure}
\begin{center}
\includegraphics[width=0.6\textwidth]{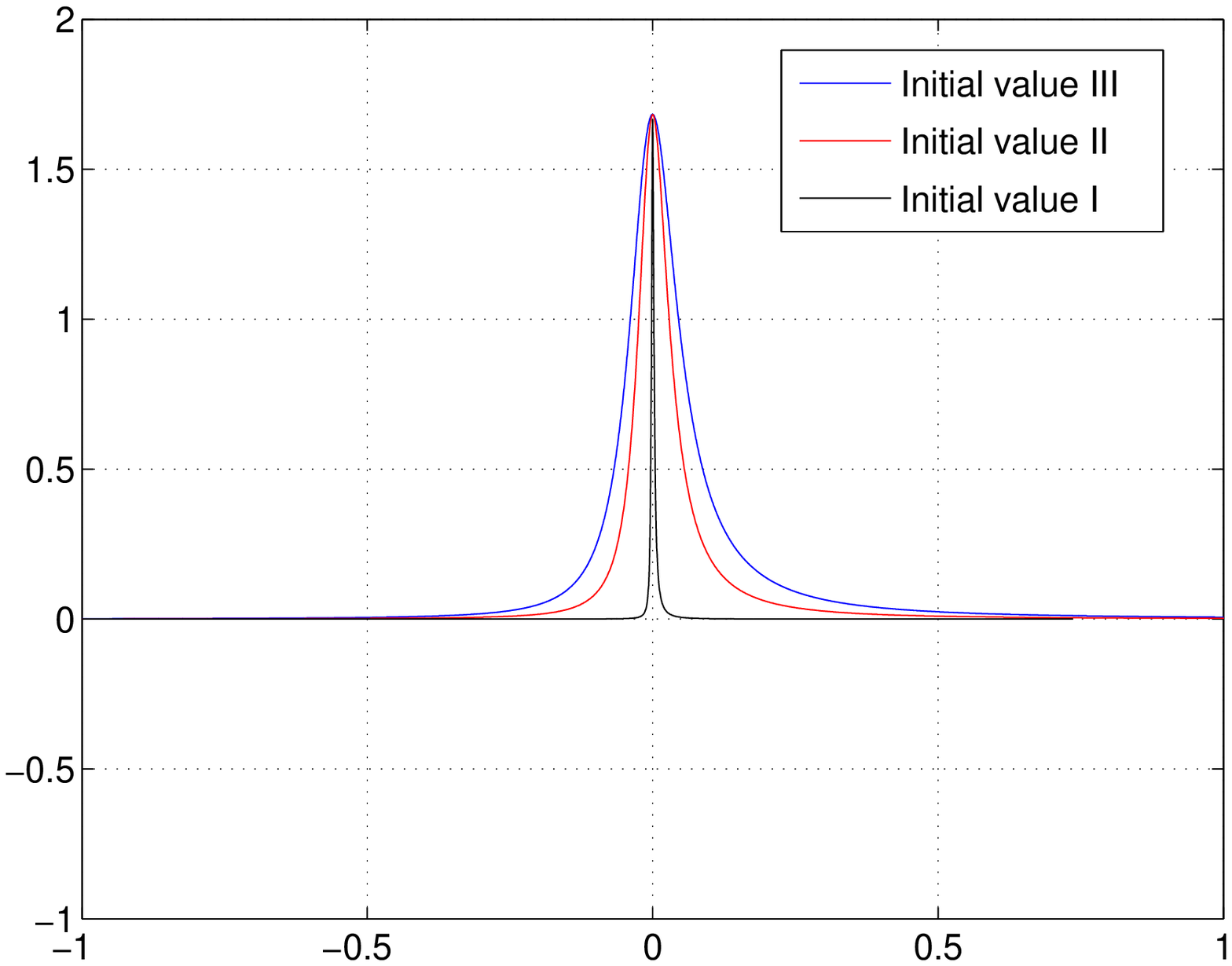}
\includegraphics[width=0.6\textwidth]{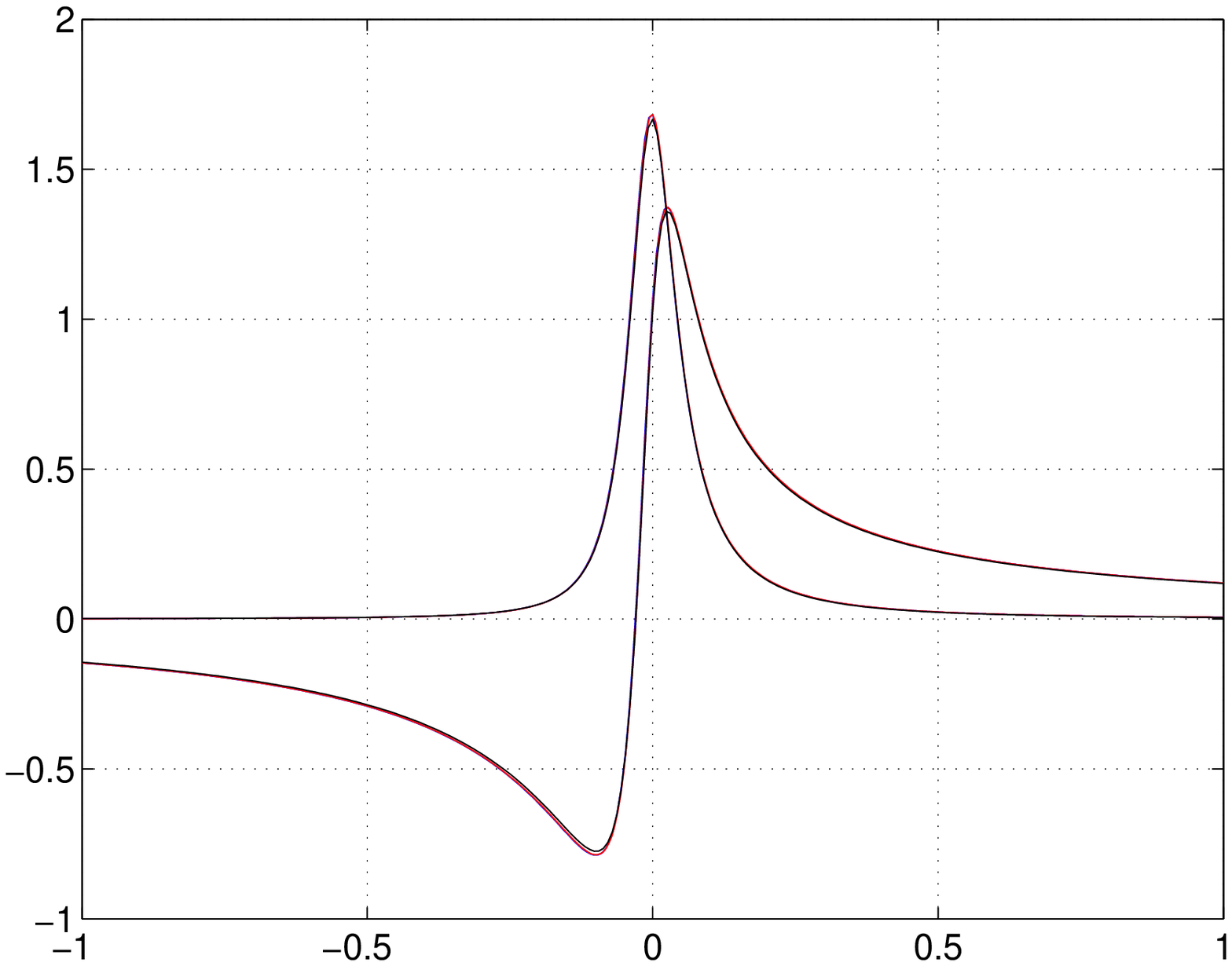}
\end{center}
\caption{The self-similar profiles for Initial Conditions I-III 
respectively ($\nu = 0$). 
Top: The original profiles for $u$ for Initial Conditions I-III; 
Bottom: The rescaled profiles. 
Black: Initial Condition I; Red: Initial Condition II; Blue: 
Initial Condition III.}
\label{profile compare invis}
\end{figure}
To gain some insight into this phenomenon, we perform some 
analysis of the self-similar solutions. We assume that the self-similar
profiles converge to a steady state as $t\rightarrow T$. 
\begin{eqnarray}
u(x,t)\rightarrow\frac{1}{T-t}\,U\left(\xi; \lambda\right),\quad
\mbox{as} \; t\rightarrow T\\
v(x,t)\rightarrow\frac{1}{T-t}\,V\left(\xi; \lambda\right),\quad
\mbox{as} \; t\rightarrow T,
\end{eqnarray}
where $\D\lambda=\lim_{t\rightarrow T}\left((T-t)^{1/2}
\frac{d}{dt}x_0(t)\right)$.

If we neglect the logarithmic correction in $\xi$ and substitute 
above equations into the nonlocal system, we obtain the equations for 
$U$ and $V$ as follows:
\begin{eqnarray}
\label{ss u}
U+\beta\xi U_\xi-\lambda U_\xi &=&UV,\\
\label{ss v}
V+\beta\xi V_\xi-\lambda V_\xi &=&H(U^2).
\end{eqnarray}

Let $U_1(\xi), V_1(\xi)$ be the solution of the self-similar 
system \myref{ss u}, \myref{ss v} corresponding to $\lambda=1$, 
then the solution for $\lambda\ne 1$ can be obtained by using the
following rescaling of the self-similar variable $\xi$:
\begin{eqnarray}
U(\xi;\lambda)&=&U_1(\lambda^{-1}\xi),\\
V(\xi;\lambda)&=&V_1(\lambda^{-1}\xi).
\end{eqnarray}
The profiles that are obtained from different initial conditions
have different $\lambda$, but they can match each other by 
rescaling $\xi$. This may explain why we can match different
rescaled profiles corresponding to different initial conditions
by rescaling $\xi$.

\subsection{Asymptotically self-similar blowup of the viscous nonlocal system}

In this subsection, we perform computations to investigate the
finite time singularity of the viscous nonlocal system.
In our computations, we choose the viscosity coefficient to be
$\nu=0.001$. Notice that the solution of the viscous nonlocal system
can not keep 
the compact support, so we only perform our numerical study for 
Initial Conditions II and III which are periodic. The 
computational settings are the same as those in the inviscid case. 

We use the same asymptotic singularity form fit as in the inviscid
model, i.e.
\begin{eqnarray}
\label{umax fit form vis}
\|u\|_{L^\infty}=\frac{C}{(T-t)^\al},
\end{eqnarray}
where $T$ is the blowup time.
In Figure \ref{umax vis ini1} and Figure \ref{umax vis ini2}, we plot 
$\|u\|_{L^\infty}^{-1}$ versus the asymptotic singularity fit. We can 
see that as we increase resolutions from $N=2^{14}$ to $N=2^{18}$,
$\|u\|_{L^\infty}^{-1}$ converges to the asymptotic fit which is
almost a perfect straight line. This suggests that $\alpha = 1$. 
From these numerical 
results, we can see that adding viscosity with $\nu=0.001$ does not
prevent the solution from blowing up and does not change the
qualitative nature of the singular solution, although it postpones 
the blowup time. 

Next, we study the rescaled profiles of the asymptotically 
self-similar solutions of the viscous nonlocal system. 
We look for a dynamically rescaled profile
$U(\xi,t), V(\xi,t)$ near the singularity of the form
\begin{eqnarray}
\label{u-self-similar-vis}
u(x,t)=\frac{1}{T-t}\,U\left(\frac{x-x_0(t)}{(T-t)^\beta},t\right),\quad 
\mbox{as} \; t\rightarrow T,\\
\label{v-self-similar-vis}
v(x,t)=\frac{1}{T-t}\,V\left(\frac{x-x_0(t)}{(T-t)^\beta},t\right),\quad
\mbox{as} \;t\rightarrow T,
\end{eqnarray}
where $T$ is the predicted blowup time in the singularity form
fit \myref{umax fit form vis}, $\beta$ is a parameter to be determined,
and $x_0(t)$ is the location in which $|u|$ achieves its
global maximum at $t$.
Again, we use a least square fit to determine $\beta$
and find that $\beta =1/2$ with a logarithmic correction.
In Figures \ref{profile vis ini1} and
\ref{profile vis ini2}, we plot the rescaled profiles of the
asymptotically self-similar solution for Initial Conditions II 
and III respectively. The dynamically rescaled 
variable $\xi$ has the same form as that of the inviscid nonlocal 
system, i.e.
\begin{eqnarray}
\xi=\frac{x-x_0(t)}{(T-t)^{1/2}\log(1/(T-t))^{1/2}}.
\label{xi-ss-form-vis}
\end{eqnarray}
In Figure \ref{profile vis ini1}, we plot the
self-similar profiles $U$ and $V$ at three
different times for Initial Condition II. We can see
that the rescaled profiles for these three different times
agree with one another very well. As in the inviscid case,
we observe that there is a significant overlap between the inner
region of $U$ and the inner region of $V$ where $V$ is
positive. Such overlap persists dynamically and is
responsible for producing a quadratic nonlinearity in
the right hand side of the $u$-equation. 
Similar observation can be made for the self-similar
profiles for Initial Condition III, see Figure 
\ref{profile vis ini2}. 

As we can see from Figures \ref{profile vis ini1} and
\ref{profile vis ini2}, the rescaled profiles of the
viscous nonlocal system is qualitatively similar to
those of the inviscid nonlocal systems. This is to be 
expected since there is a logarithmic correction in the 
rescaling variable $\xi$ in the inviscid nonlocal system. 
Consequently, the viscous term can not dominate the nonlinear 
term in the nonlocal system. On the other hand, we observe
that the profiles corresponding to the viscous nonlocal system 
are wider and more symmetric than those corresponding to the 
inviscid nonlocal system. This seems to make sense because
the viscosity tends to smooth the singularity and make the 
profiles smoother and more symmetric. 

We have also performed a similar numerical study of the
viscous nonlocal system with $\nu = 0.01$ for Initial
Conditions II and III. We find that the viscous nonlocal 
system develops an asymptotically self-similar singularity 
in a finite time with the same blowup rate and self-similar 
scaling as the case of $\nu =0.001$.

\begin{figure}
\begin{center}
\includegraphics[width=0.6\textwidth]{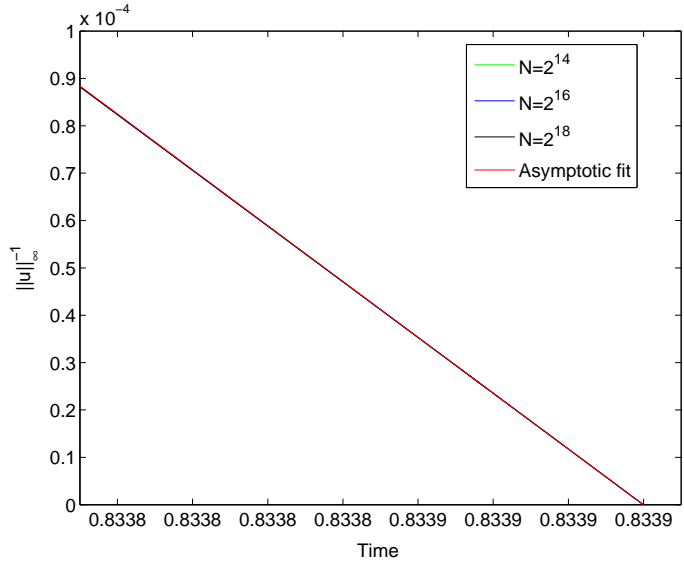}
\end{center}
\caption{The inverse of $\|u\|_{\infty}$ (black) versus the asymptotic fit (red) for Initial
Condition II with viscosity $\nu = 0.001$. The fitted blowup time is $T = 0.833919962702315$
and the scaling constant is $C = 1.69630372479547$.}
\label{umax vis ini1}
\end{figure}

\begin{figure}
\begin{center}
\includegraphics[width=0.6\textwidth]{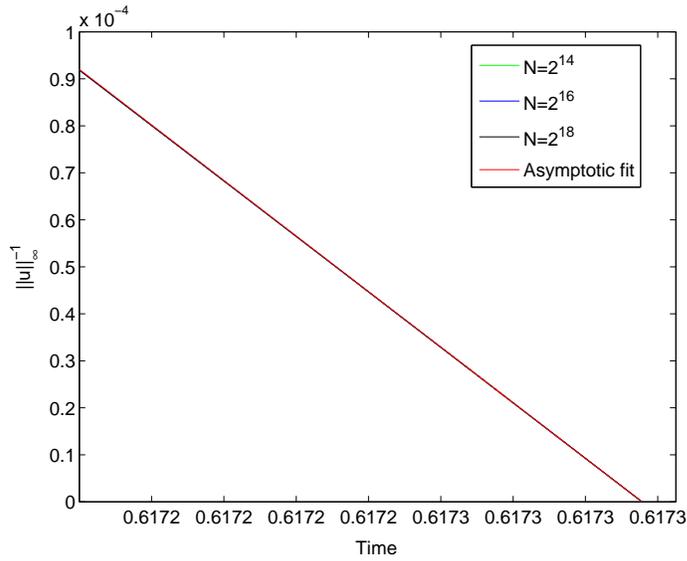}
\end{center}
\caption{The inverse of $\|u\|_{\infty}$ (black) versus the asymptotic fit (red) for Initial
Condition III with viscosity $\nu = 0.001$. The fitted blowup time is $T = 0.617315651741129$
and the scaling constant is $C = 1.69150344092375$.}
\label{umax vis ini2}
\end{figure}

\begin{figure}
\begin{center}
\includegraphics[width=0.6\textwidth]{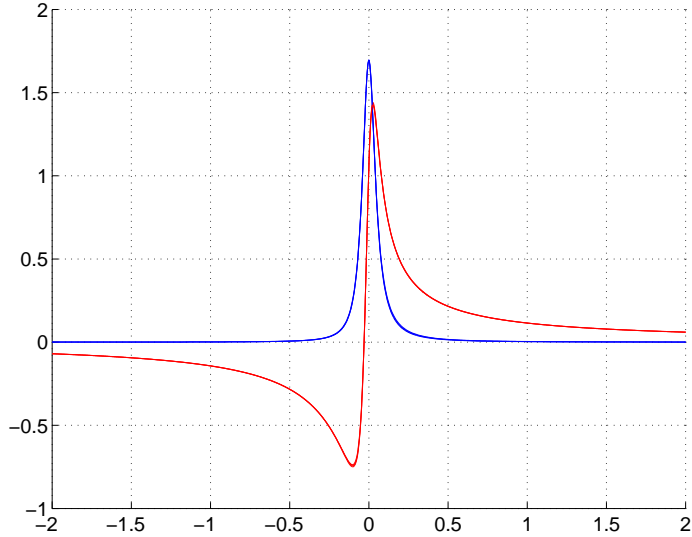}
\end{center}
\caption{Rescaled profiles $U$ and $V$ for Initial Condition II with viscosity $
\nu = 0.001$ at
$t = 0.833917828434707, 0.83391976141767$ and 0.833919943745501
 respectively. The corresponding
maximum values of u are 794399, 8416207 and 89496701 respectively. Blue: profile
of $u$; Red: profile of $v$.}
\label{profile vis ini1}
\end{figure}

\begin{figure}
\begin{center}
\includegraphics[width=0.6\textwidth]{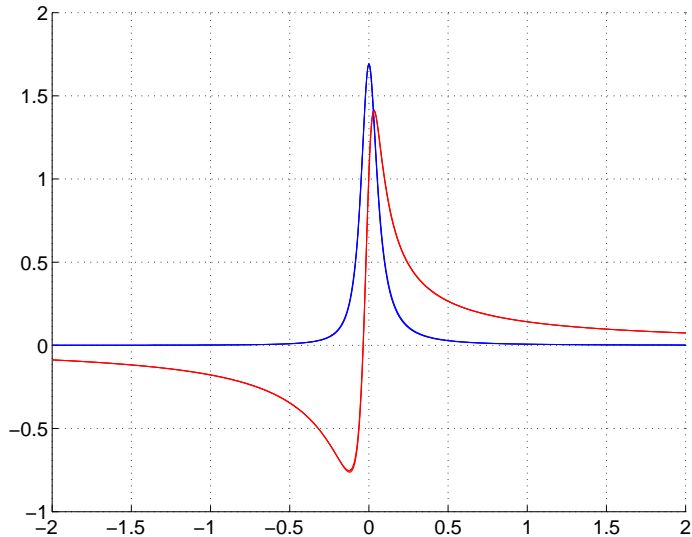}
\end{center}
\caption{Rescaled profiles $U$ and $V$ for Initial Condition III with viscosity $\nu = 0.001$ at
$t = 0.617313605456105, 0.617315459830455$ and 0.617315633726175
 respectively. The corresponding
maximum values of u are 826395, 8808734 and 94072100 respectively. Blue: profile
of $u$; Red: profile of $v$.}
\label{profile vis ini2}
\end{figure}

\vspace{0.1in}
\newpage
\noindent
{\bf Acknowledgments}
Dr. T. Hou would like to acknowledge NSF for their generous 
support through the Grants DMS-0713670 and DMS-0908546. 
The work of Drs. Z. Shi and S. Wang was supported in part by
the NSF grant DMS-0713670.
The research of Dr. C. Li was in part supported by 
the NSF grant DMS-0908546.
The research of Dr. S. Wang was supported by
the Grants NSFC 10771009 and PHR-IHLB 200906103.
The research of Dr. X. Yu was in part supported by the
Faculty of Science start-up fund of University of Alberta, 
and the research grant from NSERC. This 
work was done during Drs. Li, Wang, and Yu's visit to ACM 
at Caltech.  They would like to thank Prof. T. Hou and Caltech 
for their hospitality during their visit. Finally, we
would like to thank the anonymous referee for the
valuable comments and suggestions.

\end{document}